\documentclass[12pt]{amsart}%
\usepackage{amsmath}
\usepackage{amsfonts}
\usepackage{amssymb}
\usepackage{graphicx}
\usepackage{comment}
\usepackage[margin= 1.3 in]{geometry}
\providecommand{\U}[1]{\protect\rule{.1in}{.1in}}
\newtheorem{theorem}{Theorem}

\newtheorem{algorithm}[theorem]{Algorithm}

\newtheorem{lemma}[theorem]{Lemma}

\newtheorem{proposition}[theorem]{Proposition}
\newtheorem{remark}[theorem]{Remark}

\allowdisplaybreaks[1]

\theoremstyle{definition}
\numberwithin{equation}{section}

\newcommand{\resumename}{R\'esum\'e}

\begin{document}
\date{\today}
\title[Decomposition]{Computing Vergne Polarizing Subalgebras}
\author[V. Oussa]{Vignon Oussa}
\address{(4): Dept.\ of Mathematics\\
Bridgewater State University\\
Bridgewater, MA 02324 U.S.A.\\}
\email{Vignon.Oussa@bridgew.edu}

\maketitle

\begin{abstract}
According to the orbit method, the construction of a unitary irreducible
representation of a nilpotent Lie group requires a precise computation of some
polarizing subalgebra subordinated to a linear functional in the linear dual
of the corresponding Lie algebra. This important step is generally challenging
from a computational viewpoint. In this paper, we provide an algorithmic
approach to the construction of the well-known Vergne polarizing subalgebras
\cite{Corwin}. The algorithms presented in this paper are specifically
designed so that they can be implemented in Computer Algebra Systems. We also
show there are instances where Vergne's construction could be refined for the
sake of efficiency. Finally, we adapt our refined procedure to free nilpotent
finite-dimensional Lie algebras of step-two to obtain simple and precise
descriptions of Vergne polarizing algebras corresponding to all linear
functionals in a dense open subset of the linear dual of the corresponding
Lie algebra. Also, a
program written for Mathematica is presented at the end of the paper.

\end{abstract}




\section{Introduction}

The Lie brackets of nilpotent Lie algebras have some natural combinatorial
structures which make this class of algebras very appealing; both from
theoretical and computational aspects. The main purpose of this paper is to
reconcile some of the theoretical nature of harmonic analysis on
nilpotent Lie groups with some of the computations involved in the process of
constructing unitary irreducible representations.  It is worth noticing that Pedersen has written several programs in REDUCE to compute polarizing subalgebras and canonical coordinates for all nilpotent Lie algebras of dimensions less than seven. In fact, Michel Duflo sent us copies of a companion manuscript to \cite{Pedersen} which contains outputs of his programs. Since Pedersen's programs were written a while ago, it is appropriate that we reintroduce some of his ideas through some more current technology. Let $\mathfrak{g}$ be a real
nilpotent Lie algebra. We denote by $G$ the connected simply connected Lie
group corresponding to $\mathfrak{g}$ such that $G=\exp\left(  \mathfrak{g}%
\right)  .$ In this case, $G$ is called a nilpotent Lie group. The problem of
classification of the irreducible unitary representations of $G$ is well
understood \cite{Corwin}. Let $\mathfrak{p}$ be a subalgebra of $\mathfrak{g}$
and let $\ell\in\mathfrak{g}^{\ast}$. We say that $\mathfrak{p=p}\left(
\ell\right)  $ is subordinated to $\ell$ if $
\ell\left[  \mathfrak{p}\left(  \ell\right)  \mathfrak{,p}\left(  \ell\right)
\right]  =0.$ As a result, the formula $
\chi_{\ell}\left(  \exp X\right)  =\exp\left(  2\pi i\ell\left(  X\right)
\right)$ defines a unitary character on $\exp\left(  \mathfrak{p}\left(  \ell\right)
\right)  $. The induced representation $\mathrm{Ind}_{\exp\mathfrak{p}\left(
\ell\right)  }^{G}\left(  \chi_{\ell}\right) $ is an irreducible
representation acting in the Hilbert space $L^{2}\left(  \frac{G}%
{\exp\mathfrak{p}\left(  \ell\right)  }\right)$ if and only if
$\mathfrak{p}\left(  \ell\right)  $ is a \textbf{maximal isotropic subspace}
for the skew-symmetric bilinear form $B_{\ell}$ defined by $B_{\ell}\left(
X,Y\right)  =\ell\left[  X,Y\right]  .$ Such algebra is called a
\textbf{polarizing subalgebra} \cite{Corwin, Plancherel} for the linear
functional $\ell$ and, it is well-known that the condition that $\mathfrak{p}%
\left(  \ell\right)  $ is a maximal isotropic subalgebra is equivalent to

\begin{enumerate}
\item $\left[  \mathfrak{p}\left(  \ell\right)  ,\mathfrak{p}\left(
\ell\right)  \right]  \subset\ker\left(  \ell\right)  $

\item $\dim\mathfrak{p}\left(  \ell\right)  =n-d$ where $
\dim\left(  G\cdot\ell\right)  =2d.$
\end{enumerate}
The notation $\cdot$ is the coadjoint action of $G$ on $\mathfrak{g}^{\ast}$
which is defined as follows:%
\[
\exp X\cdot\ell\left(  Y\right)  =\ell\left(  Ad_{\exp\left(  -X\right)
}Y\right)  =\ell\left(  \exp\left(  ad_{-X}\right)  Y\right)  .
\]
One of the complications of the representation theory of
nilpotent Lie groups is that the algebra $\mathfrak{p}\left(  \ell\right)  $
is not generally unique, and there is no canonical way of constructing it.
However, there exist well-known procedures among which Vergne's construction
(\cite{Corwin} Theorem $1.3.5)$ is probably the most known. The work presented
in this paper is two-fold. In the second section of this paper, we revisit
Vergne's construction and we show that under specific conditions, the
description of Vergne is refinable. In the third section, we adapt our results
to the class of finite-dimensional free nilpotent Lie algebras of step-two. We
are then able to derive very simple descriptions of Vergne polarizing
subalgebras for all linear functionals in a dense open subset of
$\mathfrak{g}^{\ast}.$ At the end of the paper, we also provide a program written in Mathematica to compute the polarizing subalgebra of any given nilpotent Lie algebra. 

\section{Polarizing Algebras}

Let $Z_{1},Z_{2},\cdots,Z_{n}$ be a fixed strong Malcev basis \cite{Corwin} for the Lie
algebra $\mathfrak{g}$ passing through the center of $\mathfrak{g.}$ Let
\[
\left(  0\right)  \subseteq\mathfrak{g}_{1}\subseteq\mathfrak{g}_{2}%
\subseteq\cdots\subseteq\mathfrak{g}_{n}%
\]
be a chain of ideals in $\mathfrak{g}$ such that $\dim\mathfrak{g}_{j}=j.$
Given $\ell\in\mathfrak{g}^{\ast},$ let $\ell_{j}=\ell|\mathfrak{g}_{j}$ where
$\ell_{j}$ is the restriction of the linear functional $\ell$ to the vector
space $\mathfrak{g}_{j}.$ Then, the polarizing subalgebra subordinated to the
linear functional $\ell$ is described as follows
\[
\mathfrak{p}\left(  \ell\right)  =%
{\displaystyle\sum\limits_{j=1}^{n}}
\mathfrak{r}\left(  \ell_{j}\right)
\]
where $\mathfrak{r}\left(  \ell_{j}\right)  $ is the radical of the
skew-symmetric bilinear form $B_{\ell_{j}}.$ By definition
\begin{equation}
\mathfrak{r}\left(  \ell_{j}\right)  =\left\{  Y\in\mathfrak{g}_{j}%
\mathfrak{:}\text{ }B_{\ell_{j}}\left(  X,Y\right)  =0\text{ for all }%
X\in\mathfrak{g}_{j}\right\}.  \label{rad}%
\end{equation}
A direct consequence of (\ref{rad}) is that
\begin{equation}
\mathfrak{p}\left(  \ell\right)  ={\displaystyle\sum\limits_{j=1}^{n}}\text{\textrm{nullspace}}\left(  M\left(  \ell_{j}\right)  \right)
\label{pol}%
\end{equation}
where $M\left(  \ell_{j}\right)  $ is a skew-symmetric matrix of order $j$
given by
\[
M\left(  \ell_{j}\right)  =\left[
\begin{array}
[c]{cccc}%
0 & \ell_{j}\left(  \left[  Z_{1},Z_{2}\right]  \right)   & \cdots & \ell
_{j}\left(  \left[  Z_{1},Z_{j}\right]  \right)  \\
\ell_{j}\left(  \left[  Z_{2},Z_{1}\right]  \right)   & 0 & \cdots & \ell
_{j}\left(  \left[  Z_{2},Z_{j}\right]  \right)  \\
\vdots & \vdots & \ddots & \vdots\\
\ell_{j}\left(  \left[  Z_{j-1},Z_{1}\right]  \right)   & \ell_{j}\left(
\left[  Z_{j-1},Z_{2}\right]  \right)   & \cdots & \ell_{j}\left(  \left[
Z_{j-1},Z_{j}\right]  \right)  \\
\ell_{j}\left(  \left[  Z_{j},Z_{1}\right]  \right)   & \ell_{j}\left(
\left[  Z_{j},Z_{2}\right]  \right)   & \cdots & 0
\end{array}
\right]  .
\]
It is worth mentioning that for Lie algebras of arbitrarily large dimensions,
it is a very difficult and tedious task to compute Formula \ref{rad} by hands.
Therefore, there is a need to provide methods for the construction of Vergne
polarizing subalgebras that can be implemented in Computer Algebra Systems and
on other computational platforms. From a computational point of view, Formula
\ref{pol} is slightly more appealing because of the use of linear algebra
terms. Moreover, we should point out that the null-space of $M\left(  \ell
_{j}\right)  $ is obtained with respect to the fixed strong Malcev basis $Z_{1},Z_{2},\cdots,Z_{j}$ for the Lie algebra $\mathfrak{g.}$

Let us now provide an algorithm for the construction of a Vergne polarizing
subalgebra subordinated to an arbitrarily linear functional $\ell\in
\mathfrak{g}^{\ast}.$

\begin{algorithm}
\label{alg}Let
\[
\ell\in\mathfrak{g}^{\ast}.
\]
First, we set
\begin{equation}
\mathbf{M}\left(  \ell\right)  =\left[
\begin{array}
[c]{ccccc}%
0 & \ell\left[  Z_{1},Z_{2}\right]  & \cdots & \ell\left[  Z_{1}%
,Z_{n-1}\right]  & \ell\left[  Z_{1},Z_{n}\right] \\
-\ell\left[  Z_{1},Z_{2}\right]  & 0 & \cdots & \ell\left[  Z_{2}%
,Z_{n-1}\right]  & \ell\left[  Z_{2},Z_{n}\right] \\
\vdots & \vdots & \ddots & \vdots & \vdots\\
-\ell\left[  Z_{1},Z_{n-1}\right]  & -\ell\left[  Z_{2},Z_{n-1}\right]  &
\cdots & 0 & \ell\left[  Z_{n-1},Z_{n}\right] \\
-\ell\left[  Z_{1},Z_{n}\right]  & -\ell\left[  Z_{2},Z_{n}\right]  & \cdots &
-\ell\left[  Z_{n-1},Z_{n}\right]  & 0
\end{array}
\right]  \label{MatrixM}%
\end{equation}
such that $\mathbf{M}\left(  \ell\right)  $ is a singular skew-symmetric
matrix of order $n.$ Next, we define a submatrix of $\mathbf{M}\left(
\ell\right)  $ of order $j$ as follows:
\[
\mathbf{M}_{j}\left(  \ell\right)  =\left[
\begin{array}
[c]{ccccc}%
0 & \ell\left[  Z_{1},Z_{2}\right]  & \cdots & \ell\left[  Z_{1}%
,Z_{j-1}\right]  & \ell\left[  Z_{1},Z_{j}\right] \\
-\ell\left[  Z_{1},Z_{2}\right]  & 0 & \cdots & \ell\left[  Z_{2}%
,Z_{j-1}\right]  & \ell\left[  Z_{2},Z_{j}\right] \\
\vdots & \vdots & \ddots & \vdots & \vdots\\
-\ell\left[  Z_{1},Z_{j-1}\right]  & -\ell\left[  Z_{2},Z_{j-1}\right]  &
\cdots & 0 & \ell\left[  Z_{j-1},Z_{j}\right] \\
-\ell\left[  Z_{1},Z_{j}\right]  & -\ell\left[  Z_{2},Z_{j}\right]  & \cdots &
-\ell\left[  Z_{j-1},Z_{j}\right]  & 0
\end{array}
\right]  .
\]
Finally
\[
\mathfrak{p}\left(  \ell\right)  =\sum_{k=1}^{n}\mathrm{nullspace}\left(
\mathbf{M}_{j}\left(  \ell\right)  \right)  .
\]

\end{algorithm}

\begin{remark}
Although, it is fairly easy to implement Algorithm \ref{alg}, it is not
generally the most efficient way to compute the algebra $\mathfrak{p}\left(
\ell\right)  .$ In the remainder of this section, we will provide a refined
version of Algorithm \ref{alg}.
\end{remark}

Let $\mathfrak{z}\left(  \mathfrak{g}\right)  $ be the central ideal for the
Lie algebra $\mathfrak{g.}$

\begin{lemma}
If $j\leq\dim\mathfrak{z}\left(  \mathfrak{g}\right)  +1$ then
$\mathrm{nullspace}\left(  M\left(  \ell_{j}\right)  \right)  =%
\mathbb{R}
Z_{1}+\cdots+%
\mathbb{R}
Z_{j}.$
\end{lemma}

\begin{proof}
This lemma follows from the fact that if we assume that $j\leq\dim
\mathfrak{z}\left(  \mathfrak{g}\right)  +1$ then the matrix $M\left(
\ell_{j}\right)  $ is simply a zero matrix of order $j$ which we consider as a
linear operator acting on the vector space $
\mathfrak{g}_{j}=%
\mathbb{R}
Z_{1}+\cdots+%
\mathbb{R}
Z_{j}.$ Thus its null-space is equal to the whole vector space $\mathfrak{g}_{j}.$
\end{proof}

Now, let us assume that
\[
j>\dim\mathfrak{z}\left(  \mathfrak{g}\right)  +1.
\]
Then the matrix $M\left(  \ell_{j}\right)  $ is equal to%
\begin{equation}
\left[
\begin{array}
[c]{ccccccc}%
0 & \cdots & 0 & 0 & 0 & \cdots & 0\\
\vdots &  & \vdots & \vdots & \vdots &  & \vdots\\
0 & \cdots & 0 & 0 & 0 & \cdots & 0\\ 
0 & \cdots & 0 & 0 & \ell_{j}\left[  Z_{\dim\mathfrak{z}\left(  \mathfrak{g}%
\right)  +1},Z_{\dim\mathfrak{z}\left(  \mathfrak{g}\right)  +2}\right]   &
\cdots & \ell_{j}\left[  Z_{\dim\mathfrak{z}\left(  \mathfrak{g}\right)
+1},Z_{j}\right]  \\
0 & \cdots & 0 & -\ell_{j}\left[  Z_{\dim\mathfrak{z}\left(  \mathfrak{g}%
\right)  +1},Z_{\dim\mathfrak{z}\left(  \mathfrak{g}\right)  +2}\right]   &
0 &  & \ell_{j}\left[  Z_{\dim\mathfrak{z}\left(  \mathfrak{g}\right)
+2},Z_{j}\right]  \\
\vdots &  & \vdots & \vdots & \vdots & \ddots & \vdots\\
0 & \cdots & 0 & -\ell_{j}\left[  Z_{\dim\mathfrak{z}\left(  \mathfrak{g}%
\right)  +1},Z_{j}\right]   & -\ell_{j}\left[  Z_{\dim\mathfrak{z}\left(
\mathfrak{g}\right)  +2},Z_{j}\right]   & \cdots & 0
\end{array}
\right]  .\label{Me}%
\end{equation}
We will consider the following submatrix of $M\left(  \ell_{j}\right)  $ which
we denote by
\[
M_{0}\left(  \ell_{j}\right)  =\left[
\begin{array}
[c]{cccc}%
0 & \ell_{j}\left[  Z_{\dim\mathfrak{z}\left(  \mathfrak{g}\right)
+1},Z_{\dim\mathfrak{z}\left(  \mathfrak{g}\right)  +2}\right]   & \cdots &
\ell_{j}\left[  Z_{\dim\mathfrak{z}\left(  \mathfrak{g}\right)  +1}%
,Z_{j}\right]  \\
-\ell_{j}\left[  Z_{\dim\mathfrak{z}\left(  \mathfrak{g}\right)  +1}%
,Z_{\dim\mathfrak{z}\left(  \mathfrak{g}\right)  +2}\right]   & 0 &  &
\ell_{j}\left[  Z_{\dim\mathfrak{z}\left(  \mathfrak{g}\right)  +2}%
,Z_{j}\right]  \\
\vdots & \vdots & \ddots & \vdots\\
-\ell_{j}\left[  Z_{\dim\mathfrak{z}\left(  \mathfrak{g}\right)  +1}%
,Z_{j}\right]   & -\ell_{j}\left[  Z_{\dim\mathfrak{z}\left(  \mathfrak{g}%
\right)  +2},Z_{j}\right]   & \cdots & 0
\end{array}
\right]
\]
where $j=\dim\mathfrak{z}\left(  \mathfrak{g}\right)  +s_{j},$%
\begin{equation}
M\left(  \ell_{j}\right)  =\left[
\begin{array}
[c]{cc}%
0 & 0\\
0 & M_{0}\left(  \ell_{j}\right)
\end{array}
\right]  \label{block}%
\end{equation}
and $s_{j}>1.$ We observe that if $j_1\leq j_2$ then $M_0(\ell_{j_1})$ is a submatrix of $M_0(\ell_{j_2}).$ It is fairly easy to check that the above implies that if $j=\dim\left(
\mathfrak{z}\left(  \mathfrak{g}\right)  \right)  +s_{j}$ then
\[
\mathrm{nullspace}\left(  M\left(  \ell_{\dim\left(  \mathfrak{z}\left(
\mathfrak{g}\right)  \right)  +s_{j}}\right)  \right)  =\mathfrak{z}\left(
\mathfrak{g}\right)  +\mathrm{nullspace}\left(  M_{0}\left(  \ell_{\dim\left(
\mathfrak{z}\left(  \mathfrak{g}\right)  \right)  +s_{j}}\right)  \right)  .
\]
Now, let
\[
\mathbf{I}\left(  \ell\right)  =\mathbf{I=}\left\{  1<s_{j}\leq n-\dim\left(
\mathfrak{z}\left(  \mathfrak{g}\right)  \right)  :\mathrm{rank}\left(
M_{0}\left(  \ell_{\dim\left(  \mathfrak{z}\left(  \mathfrak{g}\right)
\right)  +s_{j}}\right)  \right)  =s_{j}\right\}  \subset\left\{
2,3,\cdots,\dim\left(  \mathfrak{g}\right)  \right\}
\]
Clearly, it is possible for the set $\mathbf{I}$ to be empty. However, in the case where $\mathbf{I}$ is nonempty, it makes sense
to attempt to refine Formula \ref{pol}.

\begin{lemma}
If $s_{j}\in\mathbf{I}$ then
\[
\mathrm{nullspace}\left(  M\left(  \ell_{j}\right)  \right)
=\mathrm{nullspace}\left(  M\left(  \ell_{\dim\mathfrak{z}\left(
\mathfrak{g}\right)  +s_{j}}\right)  \right)  =\mathfrak{z}\left(
\mathfrak{g}\right ).
\]

\end{lemma}

\begin{proof}
If $s_{j}\in\mathbf{I}$ then $M_{0}\left(  \ell_{\dim\mathfrak{z}\left(
\mathfrak{g}\right)  +s_{j}}\right)  $ is a skew-symmetric matrix of even
full-rank. Thus, its null-space is the trivial vector space. Therefore, $$
\mathrm{nullspace}\left(  M\left(  \ell_{j}\right)  \right)  =\mathfrak{z}%
\left(  \mathfrak{g}\right).$$

\end{proof}

Appealing to the above lemma, and to Formula \ref{pol}, the following is immediate.

\begin{theorem}
\label{t1}Let $\ell$ be a linear functional in $\mathfrak{g}^{\ast}.$ A Vergne
polarizing subalgebra subordinated to the linear functional $\ell$ is
\[
\mathfrak{p}\left(  \ell\right)  =\mathfrak{z}\left(  \mathfrak{g}\right)  +\mathbb{R}Z_{\mathrm{dim\mathfrak{z(g)}+1}}+{\displaystyle\sum\limits_{s_{j}\notin\mathbf{I}(\ell)}}
\mathrm{nullspace}\left(  M_{0}\left(  \ell_{\dim\mathfrak{z}\left(
\mathfrak{g}\right)  +s_{j}}\right)  \right)
\]

\end{theorem}

\section{Vergne Polarizing Subalgebras for Free Nilpotent Lie algebras}

In this section, we will adapt the algorithms provided in the previous section
to a class of nilpotent Lie algebras known as free nilpotent Lie algebras of
step-two. Applying the refined algorithms described in Theorem \ref{t1}, we
are able to provide very simple descriptions of Vergne polarizing algebras
subordinated to a family of linear functionals in a Zariski open (dense) subset of the
linear dual of the corresponding Lie algebra. Let $\mathfrak{g}$ be the free
nilpotent Lie algebra of step-two on $m$ generators $\left(  m>1\right)  $. If
$\left\{  Z_{1},\cdots,Z_{m}\right\}  $ is the generating set for
$\mathfrak{g}$ then 
\[
\mathfrak{g}=\mathfrak{z}\left(  \mathfrak{g}\right)  +%
\mathbb{R}
\text{-span }\left\{  Z_{1},\cdots,Z_{m}\right\}
\]
such that $\mathfrak{z}\left(  \mathfrak{g}\right)  =\text{ }%
\mathbb{R}
\text{-span }\left\{  Z_{ik}:1\leq i\leq m\text{ and }i<k\leq m\right\}  .$
The non-trivial Lie brackets of this Lie algebra are described as follows.
\[
\left[  Z_{i},Z_{j}\right]  =Z_{ij\text{ }}\text{for }\left(  1\leq i\leq
m\text{ and }i<j\leq m\right)  .
\]
It is then easy to see that
\[
\dim\left(  \mathfrak{z}\left(  \mathfrak{g}\right)  \right)  =\frac{m\left(
m-1\right)  }{2}.
\]
In these settings, we define recursively the matrix
\[
M_{0}\left(  \ell_{j}\right)  =\left[
\begin{array}
[c]{ccccc}%
0 & \ell_{j}\left[  Z_{1},Z_{2}\right]   & \cdots & \ell_{j}\left[
Z_{1},Z_{j-1}\right]   & \ell_{j}\left[  Z_{1},Z_{j}\right]  \\
\ell_{j}\left[  Z_{2},Z_{1}\right]   & 0 & \cdots & \ell_{j}\left[
Z_{2},Z_{j-1}\right]   & \ell_{j}\left[  Z_{2},Z_{j}\right]  \\
\vdots & \vdots & \ddots & \vdots & \vdots\\
\ell_{j}\left[  Z_{j-1},Z_{1}\right]   & \ell_{j}\left[  Z_{j-1},Z_{2}\right]
& \cdots & 0 & \ell_{j}\left[  Z_{j-1},Z_{j}\right]  \\
\ell_{j}\left[  Z_{j},Z_{1}\right]   & \ell_{j}\left[  Z_{j},Z_{2}\right]   &
\cdots & \ell_{j}\left[  Z_{j},Z_{j-1}\right]   & 0
\end{array}
\right]
\]
such that
\[
M_{0}\left(  \ell_{j}\right)  =\left[
\begin{array}
[c]{cc}%
M_{0}\left(  \ell_{j-1}\right)   & v\left(  \ell_{j}\right)  \\
w\left(  \ell_{j}\right)   & 0
\end{array}
\right]
\]
where
\[
v\left(  \ell_{j}\right)  =\left[
\begin{array}
[c]{c}%
\ell_{j}\left[  Z_{1},Z_{j}\right]  \\
\ell_{j}\left[  Z_{2},Z_{j}\right]  \\
\vdots\\
\ell_{j}\left[  Z_{j-1},Z_{j}\right]
\end{array}
\right]  \text{ and }w\left(  \ell_{j}\right)  =\left[
\begin{array}
[c]{cccc}%
\ell_{j}\left[  Z_{j},Z_{1}\right]   & \ell_{j}\left[  Z_{j},Z_{2}\right]   &
\cdots & \ell_{j}\left[  Z_{j},Z_{j-1}\right]
\end{array}
\right]  .
\]
From now on, we will assume that $\ell\in\mathfrak{g}^{\ast}$ and
\begin{equation}
\ell\in\left\{  f\in\mathfrak{g}^{\ast}:\det M_{0}\left(  f_{j-1}\right)
\neq0\text{ for odd }j>1\right\}  \label{Zariski}.
\end{equation}
Let $k$ be a natural number smaller than the dimension of the Lie algebra $\mathfrak{g}.$ Furthermore, for each $k$ we define the embedding map $\mu_k:\mathbb{R}^k\to\mathfrak{g}$ such that $\mu_k$ is a map which sends the column vector $\left[z_1,\cdots,z_k\right]^{t}$ to the vector $z_1Z_1+\cdots z_k Z_k\in \mathfrak{g}$ and  $\left[z_1,\cdots,z_k\right]^{t}$ is the transpose of $\left[z_1,\cdots,z_k\right].$
\begin{lemma}
\label{l1}If $j>1$ is odd then the null-space of the matrix $M_{0}\left(
\ell_{j}\right)  $ is equal to
\[%
\mathbb{R}
\left(  Z_{j}-\mu_{j-1}\left(\left[
\begin{array}
[c]{cccc}%
0 & \ell_{j}\left[  Z_{1},Z_{2}\right]  & \cdots & \ell_{j}\left[
Z_{1},Z_{j-1}\right] \\
\ell_{j}\left[  Z_{2},Z_{1}\right]  & 0 & \cdots & \ell_{j}\left[
Z_{2},Z_{j-1}\right] \\
\vdots & \vdots & \ddots & \vdots\\
\ell_{j}\left[  Z_{j-1},Z_{1}\right]  & \ell_{j}\left[  Z_{j-1},Z_{2}\right]
& \cdots & 0
\end{array}\right]
^{-1}\left[
\begin{array}
[c]{c}%
\ell_{j}\left[  Z_{1},Z_{j}\right] \\
\ell_{j}\left[  Z_{2},Z_{j}\right] \\
\vdots\\
\ell_{j}\left[  Z_{j-1},Z_{j}\right]
\end{array}
\right]  \right)  \right)  .
\]
If $j$ is even then
\[
\mathrm{nullspace}\left(  M_{0}\left(  \ell_{j}\right)  \right)  =\left\{
0\right\}  .
\]

\end{lemma}

\begin{proof} For any linear functional $\ell$ satisfying (\ref{Zariski})), we will show that the set $\mathbf{I}\left(  \ell\right)  $ only contains even
indices. First of all, let us suppose that $j$ is odd. Then $M_{0}\left(  \ell
_{j-1}\right)$ is a skew-symmetric matrix of even \textbf{full}-rank (see
(\ref{Zariski})). So
\[
\left[
\begin{array}
[c]{cc}%
M_{0}\left(  \ell_{j-1}\right)  ^{-1} & 0\\
0 & 0
\end{array}
\right]  \left[
\begin{array}
[c]{cc}%
M_{0}\left(  \ell_{j-1}\right)  & v\left(  \ell_{j}\right) \\
w\left(  \ell_{j}\right)  & 0
\end{array}
\right]  =\left[
\begin{array}
[c]{cc}%
I & M_{0}\left(  \ell_{j-1}\right)  ^{-1}v\left(  \ell_{j}\right) \\
0 & 0
\end{array}
\right]
\]
and\
\[
\left[
\begin{array}
[c]{cc}%
I & M_{0}\left(  \ell_{j-1}\right)  ^{-1}v\left(  \ell_{j}\right) \\
0 & 0
\end{array}
\right]
\]
is just the row-reduced form of the matrix
\[
\left[
\begin{array}
[c]{cc}%
M_{0}\left(  \ell_{j-1}\right)  & v\left(  \ell_{j}\right) \\
w\left(  \ell_{j}\right)  & 0
\end{array}
\right]  .
\]
Therefore,
\[
\mathrm{nullspace}\left(  M_{0}\left(  \ell_{j}\right)  \right)  =%
\mathbb{R}
\left(  Z_{j}-\mu_{j-1}\left(M_{0}\left(  \ell_{j-1}\right)  ^{-1}v\left(  \ell_{j}\right)
\right)\right)
\]
and the above is precisely equal to
\[%
\mathbb{R}
\left(  Z_{j}-\mu_{j-1}\left(\left[
\begin{array}
[c]{cccc}%
0 & \ell_{j}\left[  Z_{1},Z_{2}\right]  & \cdots & \ell_{j}\left[
Z_{1},Z_{j-1}\right] \\
\ell_{j}\left[  Z_{2},Z_{1}\right]  & 0 & \cdots & \ell_{j}\left[
Z_{2},Z_{j-1}\right] \\
\vdots & \vdots & \ddots & \vdots\\
\ell_{j}\left[  Z_{j-1},Z_{1}\right]  & \ell_{j}\left[  Z_{j-1},Z_{2}\right]
& \cdots & 0
\end{array}\right]
^{-1}\left[
\begin{array}
[c]{c}%
\ell_{j}\left[  Z_{1},Z_{j}\right] \\
\ell_{j}\left[  Z_{2},Z_{j}\right] \\
\vdots\\
\ell_{j}\left[  Z_{j-1},Z_{j}\right]
\end{array}
\right]  \right)  \right)  .
\]
Secondly, by assumption, if $j$ is even then $M_{0}\left(  \ell_{j}\right)  $ is a skew-symmetric
matrix of even full-rank. Therefore the null-space of this matrix is trivial.
\end{proof}

\begin{proposition}
If $m$ is odd ($m=2s+1$ for some natural number $s$) then%
\[
\mathfrak{p}\left(  \ell\right)  =\mathfrak{z}\left(  \mathfrak{g}\right)
\text{ }+\text{ }%
\mathbb{R}
Z_{1}\text{ }+\left(
{\displaystyle\sum\limits_{k=1}^{s}}
\mathbb{R}
\left(  Z_{2k+1}-\mu_{2k}\left(M_{0}\left(  \ell_{2k}\right)  ^{-1}v\left(  \ell
_{2k+1}\right)  \right)  \right)\right)  .
\]
If $m$ is even ($m=2s$ for some natural number $s$) then
\[
\mathfrak{p}\left(  \ell\right)  =\mathfrak{z}\left(  \mathfrak{g}\right)
\text{ }+\text{ }%
\mathbb{R}
Z_{1}\text{ }+\left(
{\displaystyle\sum\limits_{k=1}^{s-1}}
\mathbb{R}
\left(  Z_{2k+1}-\mu_{2k}\left(M_{0}\left(  \ell_{2k}\right)  ^{-1}v\left(  \ell
_{2k+1}\right)  \right)  \right) \right) .
\]

\end{proposition}

The above proposition is a direct application of Lemma \ref{l1} and Theorem
\ref{t1}.
\section*{A Mathematica Program}
Let $\mathfrak{g}$ be a real nilpotent Lie algebra of dimension $n$ with a fixed strong Malcev basis: $Z_1,Z_2,\cdots,Z_n.$ We will present a program written in Mathematica which can be used to compute the Vergne polarizing subalgebra with respect to a linear functional $\ell\in
\mathfrak{g}^{\ast}$. The only argument for this program is the matrix
\[
\mathbf{M}\left(  \ell\right)
\]
The output of this program is a spanning set for a polarizing subalgebra subordinated to $\ell.$
\begin{figure}[!htbp]
\begin{center}
\includegraphics[scale=1]{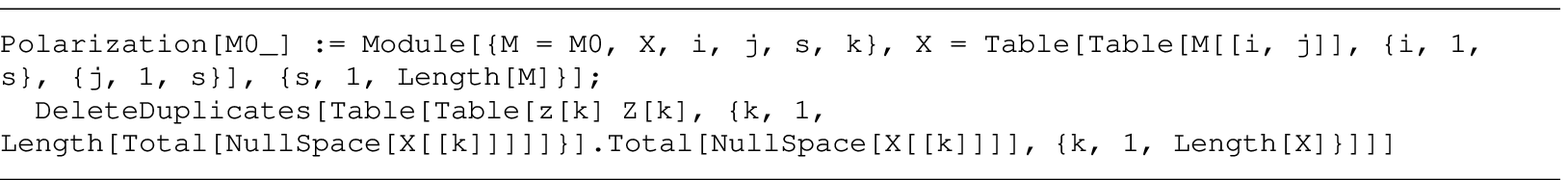}
\end{center}
\caption[A program written in Mathematica]{A program written in Mathematica}%
\label{Serie}%
\end{figure}

Here are some actual Mathematica outputs
\begin{figure}[!htbp]
\begin{center}
\includegraphics[scale=1]{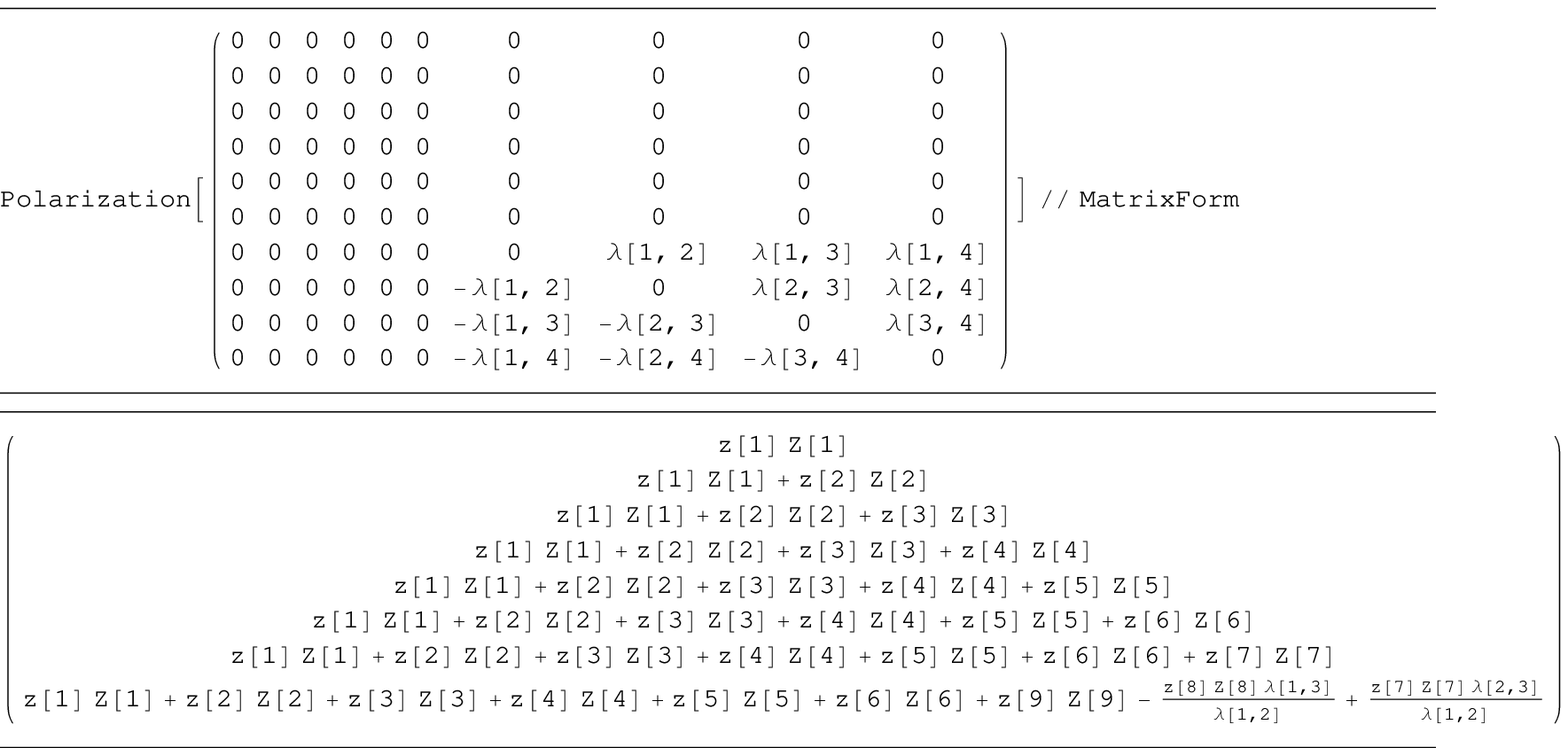}
\end{center}
\caption[Some Mathematica Outputs]{Some Mathematica Outputs}%
\label{Serie}%
\end{figure}

\section*{acknowledgment}
We thank Michel Duflo for bringing to our attention that Niels Pedersen has already written programs in REDUCE to compute polarizing subalgebras for all nilpotent Lie algebras of dimensions less than seven, and we also thank him for supplying reference \cite{Pedersen}.

\end{document}